\documentclass[11pt]{amsart}

\usepackage{enumerate, url, parskip, bm}
\usepackage{amssymb}
\usepackage{amsmath,amscd}
\usepackage{amsthm}
\usepackage{graphicx}
 \usepackage{verbatim}
 \usepackage{color}
\usepackage{amsmath}
\usepackage[utf8]{inputenc}
\usepackage[T1]{fontenc}
\usepackage{hyperref}
\usepackage{diagrams}

\usepackage{
  etex,  
  lmodern,
  calc,
  amssymb,
  amsthm,
  paralist,
  tabularx,
  multicol,
  graphicx,  
}

\usepackage[all]{xy}

\newcommand{\QQ}{\mathbb{Q}}

\newcommand{\ZZ}{\mathbb{Z}}

\newcommand{\rank}{\operatorname{rank}}

\newcommand{\bq}{/\!\!/}
\newcommand{\B}{(S^3)^2\bq T^2}

\newcommand{\Vect}{\operatorname{Vect}}
\newcommand{\K}{\operatorname{K}}
\newcommand{\KO}{\operatorname{KO}}
\newcommand{\R}{\operatorname{R}}

	\newcommand*{\defeq}{\mathrel{\vcenter{\baselineskip0.5ex \lineskiplimit0pt
				\hbox{\scriptsize.}\hbox{\scriptsize.}}}%
		=}

\newtheorem{theorem}{Theorem}[section]

\newtheorem{corollary}[theorem]{Corollary}
\newtheorem{lemma}[theorem]{Lemma}

\newtheorem*{theorem*}{Theorem}

\newtheorem{prop}[theorem]{Proposition}

\newtheorem{question}[theorem]{Question}
\theoremstyle{remark}
\newtheorem{remark}[theorem]{Remark}

\theoremstyle{definition}

\makeatletter
\def\equalsfill{$\m@th\mathord=\mkern-7mu
\cleaders\hbox{$\!\mathord=\!$}\hfill
\mkern-7mu\mathord=$}
\makeatother

\begin{document}

\abovedisplayskip=6pt plus3pt minus3pt
\belowdisplayskip=6pt plus3pt minus3pt

\title[Biquotient vector bundles with no inverse]{Biquotient vector bundles with no inverse}

\author[Jason DeVito]{Jason DeVito}
\address{The University of Tennessee at Martin, Tennessee, USA}
\email{jdevito1@utm.edu}

\author[David Gonz\'alez-\'Alvaro]{David Gonz\'alez-\'Alvaro}
\address{ETSI de Caminos, Canales y Puertos\\ Universidad Polit\'ecnica de Madrid\\ 28040 Spain}
\email{david.gonzalez.alvaro@upm.es}

\thanks{2010 \it  Mathematics Subject classification.\rm\ 
Primary 53C20.}

\thanks{The second author received support from MINECO grant MTM2017-85934-C3-2-P (Spain).  }

\begin{abstract} In previous work
, the second author and others have found conditions on a homogeneous space $G/H$ which imply that, up to stabilization, all vector bundles over $G/H$ admit Riemannian metrics of non-negative sectional curvature.  One important ingredient of their approach is Segal's result that the set of vector bundles of the form $G\times_H V$ for a representation $V$ of $H$ contains inverses within the class.  We show that this approach cannot work for biquotients $G\bq H$, where we consider vector bundles of the form $G\times_{H} V$. We call such vector bundles biquotient bundles. Specifically, we show that in each dimension $n\geq 4$ except $n=5$, there is a simply connected biquotient of dimension $n$ with a biquotient bundle which does not contain an inverse within the class of biquotient bundles.  In addition, we show that for $n\geq 6$ except $n=7$, there are infinitely many homotopy types of biquotients with the property that no non-trivial biquotient bundle has an inverse. Lastly, we show that every biquotient bundle over every simply connected biquotient $M^n = G\bq H$ with $G$ simply connected and with $n\in \{2,3,5\}$ has an inverse in the class of biquotient bundles.
\end{abstract}

\maketitle

\thispagestyle{empty}

\section{Motivation and results}

Throughout this article, $M$ will denote a closed manifold and $\Vect(M)$ the set of isomorphism classes of vector bundles over $M$. We will consider the case of real and complex vector bundles simultaneously, unless otherwise stated. It is well known that given $E\in\Vect(M)$ one can always find $F\in\Vect(M)$ such that the Whitney sum $E\oplus F$ is isomorphic to a trivial bundle, see e.g. \cite[Lemma~9.3.5]{AGP02}. We call such a bundle $F$ an \emph{inverse} for $E$. Observe that there are infinitely many inverses for a given bundle, although they all lie in the same stable class. In this article we discuss whether particularly interesting subsets $S\subset \Vect(M)$ contain their own inverses, motivated by certain K-theoretical and geometrical applications. 

As a warm up example, suppose that $M$ carries an action by a compact Lie group $G$, and consider the natural subset $\Vect_G(M)$ of $G$-vector bundles. As observed by Segal, any such bundle has an inverse in $\Vect_G(M)$, see \cite[Proposition~2.4]{Se68}. When the $G$-action is transitive, $M$ is a homogeneous space $G/H$ and $\Vect_G(G/H)$ coincides with the set of the so-called \emph{homogeneous (vector) bundles}, i.e.~those of the form $G\times_H V$ for a representation $H$ of $V$ \cite[p.~130, c)]{Se68}. Hence, every homogeneous bundle has an inverse which is also homogeneous. The purpose of this work is to show that the analogous property fails to hold, in general, for the more general class of biquotient bundles.


We first recall the definition of biquotient, following the approach by Totaro in \cite[Lemma~1.1~(3)]{To02}. Let $G$ be a compact Lie group and $Z(G)$ its center, and let $Z<G\times G$ be the diagonal normal subgroup $Z\defeq \{(g,g) : g\in Z(G)\}$. Any homomorphism $f\colon H\to (G\times G)/Z$ can be written in the form $f(h)=[f_1(h),f_2(h)]$ and determines a well-defined two-sided $H$-action $\star$ on $G$ by the rule $h\star g=f_1(h)gf_2(h)^{-1}$. When this action is free, the orbit space, denoted by $G\bq H$, inherits a manifold structure and is called a \emph{biquotient}.

By construction, associated to a biquotient $G\bq H$ there is a principal $H$-bundle $H\to G\to G\bq H$. Each representation $V$ of $H$ induces a vector bundle $G\times_H V$ over $G\bq H$, defined as the quotient of $G\times V$ via the diagonal action by $H$ consisting of the $\star$-action on $G$ and the representation action on $V$; the projection map is given by $[g,v]\mapsto [g]$. A vector bundle constructed in such a way will be called a \emph{biquotient (vector) bundle}. A biquotient bundle is real or complex depending on the nature of the representation $V$. We remark that, even though a given biquotient has several descriptions as a biquotient, when we refer to biquotient bundles over $G\bq H$ we always mean a bundle of the form $G \times_H V$.

\begin{theorem}\label{thm:biqinv}
In each dimension $n$ with $n\neq 2,3,5$, there is a simply connected biquotient $M = G\bq H$ and a biquotient bundle $G\times_H V$ with no inverse among biquotient bundles.
\end{theorem}

We stress that Theorem~\ref{thm:biqinv} equally holds for real and complex biquotient bundles. Simple examples satisfying the statement in Theorem~\ref{thm:biqinv} include the connected sum $\B\approx\mathbb{C}P^2\sharp\mathbb{C}P^2$ and Eschenburg's inhomogeneous flag manifold $SU(3)\bq T^2$, see Proposition \ref{PROP:CP2CP2_and_Eschenburg_satisfy_ast}. By taking appropriate products with certain bundles over $\mathbb{C}P^2\sharp \mathbb{C}P^2$, we can improve Theorem \ref{thm:biqinv} in dimension $n\geq 6$.


\begin{theorem}\label{thm:improve}
In each dimension $n\geq 6$ with $n\neq 7$ there are infinitely many homotopy types of simply connected biquotients $M = G\bq H$ with the following stronger property:  if $V$ is any non-trivial representation of $H$, then the biquotient bundle $G\times_H V$ does not have an inverse among biquotient bundles.  Moreover, when $n=7$ there are infinitely many homeomorphism types with this stronger property.
\end{theorem}

In fact, when $n\neq 7$, we find infinitely many pairwise non-isomorphic cohomology rings.  Our examples in dimension $7$ are certain $S^3$-bundles over $\mathbb{C}P^2\sharp \mathbb{C}P^2$ whose integral cohomology ring is isomorphic to that of $S^3\times \mathbb{C}P^2\sharp \mathbb{C}P^2$.  We distinguish them by their first Pontryagin class.


It would be interesting to derive sufficient conditions on a given closed biquotient $G\bq H$ for the existence of inverses within the class of biquotient bundles. As mentioned above, by Segal's work one such condition is that $G\bq H$ is actually a homogeneous space. Here we provide the following topological sufficient conditions.

\begin{theorem}\label{thm:biqinv_sufficient_conditions}
Let $M = G\bq H$ be a closed biquotient. If $\oplus_{i>0} H^{4i}(M,\QQ)=0$, then any real biquotient bundle has an inverse among biquotient bundles.


Suppose moreover that one of the following stronger assumptions holds:
\begin{itemize}
\item $\oplus_{i>0} H^{2i}(M,\QQ)=0$

\item $\oplus_{i>0} H^{2i}(M,\QQ)=H^{2}(M,\QQ)=\QQ$ and $M$ and $G$ are simply connected.
\end{itemize}
Then any complex biquotient bundle has an inverse among biquotient bundles.
\end{theorem}

From the classification of low-dimensional biquotients \cite{To02,De14,Pa04}, one easily sees that Theorem \ref{thm:biqinv_sufficient_conditions} applies to all simply connected biquotients in dimension $2,3,$ and $5$.  Thus, we immediately deduce the following corollary which explains the exceptions in Theorem \ref{thm:biqinv}.


\begin{corollary}\label{COR:dimensions_2_3_5}
For any presentation $G\bq H$ of a simply connected biquotient of dimension $2,3,$ or $5$ with $G$ simply connected, any biquotient bundle has an inverse among biquotient bundles.
\end{corollary}

As shown by Totaro \cite[Lemma~3.1]{To02}, every simply connected biquotient has such a presentation.

Our interest in the existence of inverses within the class of biquotient bundles arises when studying the following:

\begin{question}\label{QUEST}
Which vector bundles over a given biquotient are isomorphic to a biquotient bundle? 
\end{question}

The importance of Question~\ref{QUEST} is that the presence of a biquotient bundle structure has very nice topological and geometrical implications. On the topological side, Singhof developed a method that allows the computation of certain characteristic classes \cite{Si93}. On the geometrical side, biquotient bundles admit metrics of non-negative sectional curvature with very interesting properties, see e.g. \cite{Ta08}. 

Roughly speaking, the existence of inverses within a given subset $S\subset \Vect(M)$ allows the use of K-theory in order to detect which bundles over $M$ actually belong to $S$, up to stabilization. To make this connection precise, let $\K(M)$ and $\K(S)$ denote the Grothendieck groups of $\Vect(M)$ (i.e.~topological K-theory) and of $S$ respectively, and denote by $\imath\colon \K(S)\to \K(M)$ the induced group inclusion. Given an integer $k$, we denote also by $k$ the trivial bundle of rank $k$.






\begin{prop}\label{PROP}
Suppose that $S\subset \Vect(M)$ satisfies the following properties:
\begin{enumerate}
\item $S$ contains all trivial bundles,
\item the sum of two bundles in $S$ belongs to $S$, and
\item every bundle in $S$ has an inverse in $S$.
\end{enumerate}
Then, the following are equivalent:
\begin{itemize}
\item The natural inclusion $\imath\colon \K(S)\to \K(M)$ is surjective.
\item For every vector bundle $E$ over $M$ there exists an integer $k$ (depending on $E$) such that the Whitney sum $E\oplus k$ is isomorphic to a bundle in $S$.
\end{itemize}
\end{prop}

We remark that Proposition~\ref{PROP} and its proof holds for both complex or real vector bundles; note that in the real case the Grothendieck group of a set of vector bundles is usually denoted by $\KO(-)$ instead of $\K(-)$.
 
The upshot of Proposition~\ref{PROP} is that K-theory has been extensively investigated and there is a powerful machinery at our disposal for studying the possible surjectivity of $\imath$. This has been done in various situation as we review next, where we restrict to the case of complex bundles for simplicity. The assumptions on $S$ are satisfied when $M$ is a $G$-manifold and $S=\Vect_G(M)$; its Grothendieck group $\K(S)$ is usually denoted by $\K_G(M)$ \cite{Se68}. When $M=G/H$ is homogeneous, $\K_G(M)$ identifies naturally with the representation ring $\R(H)$. In this case, the study of the surjectivity of the map $\imath\colon \R(H)\to \K(G/H)$ was already initiated by Atiyah and Hirzebruch in the seminal article in which K-theory was founded \cite[Section~5]{AH61}. Recently, surjectivity of $\imath\colon \R(H)\to \K(G/H)$ has been characterized in the case where $H,G$ are connected and $\pi_1 G$ is torsion-free, and it turns out that it is equivalent to the condition $\rank G-\rank H\leq 1$ (see \cite[Theorem~D]{AGZ} and its proof). In the case where $M$ is a cohomogeneity one $G$-space, sufficient conditions for the surjectivity of $\imath\colon \K_G(M)\to \K(M)$ have been given in \cite[Corollary~2.2]{Ca}.

We finish by providing a geometric application which motivated the present work. Some particular cases of Proposition~\ref{PROP} were crucial in the series of papers \cite{Go17,GZ,AGZ}. Therein, the surjectivity of the map $\imath\colon \K_G(M)\to \K(M)$, for certain manifolds $M$ on which a compact Lie group $G$ acts transitively or by cohomogeneity one, is used to show that every vector bundle over $M$ (is a $G$-vector bundle and hence) carries a metric of non-negative sectional curvature, up to stabilization. This partially answers the so-called Converse Question to the Soul Theorem of Cheeger and Gromoll \cite{CG72} asking which vector bundles over a non-negatively curved closed manifold admit a complete metric of non-negative curvature. 

The set $S$ of biquotient bundles over a given $G\bq H$ satisfies the first two assumptions in Proposition~\ref{PROP}, since the rule that assigns biquotient bundles to representations commutes with respect to direct sums (and tensor products). Unfortunately, the third condition is not satisfied, in general, as Theorem~\ref{thm:biqinv} shows. Thus, Proposition~\ref{PROP} does not help to partially answer Question~\ref{QUEST} nor the Converse Question to the Soul Theorem over arbitrary biquotients. In any case, we are not aware of any general result concerning the surjectivity of the natural map $\imath\colon \R(H)\to \K( G\bq H)$.

We now describe the outline of the paper.  Section \ref{sec:Ktheory} is devoted to proving Proposition \ref{PROP} as well as Theorem \ref{thm:biqinv_sufficient_conditions}.  The main tool used is $K$-theory. Section \ref{sec:noinv} contains the proof of Theorems \ref{thm:biqinv} and \ref{thm:improve}.  
In order to find examples, we identify a cohomological property which guarantees that the only biquotient bundles with inverses are trivial (Proposition \ref{prop:noreverse}).  We then just need to construct the relevant biquotients verifying this property; we do this in Proposition \ref{PROP:RA_satisifes_ast}.

\textbf{Acknowledgments. } We would like to thank Jeffrey Carlson for helpful conversations. The second author is very grateful to Manuel Amann, Igor Belegradek, Gabino González-Diez, Marcus Zibrowius and Wolfgang Ziller for highly useful comments on a preliminary version of this manuscript.


\section{K-theoretical proofs: biquotients bundles with inverses}\label{sec:Ktheory}

Here we prove Theorem~\ref{thm:biqinv_sufficient_conditions} and Proposition~\ref{PROP}. The proofs use basics of K-theory and stable classes of vector bundles. We refer e.g.~to \cite[Section~9]{AGP02} for the reader unfamiliar with the topic.

\begin{lemma}\label{LEM:finite_Ktheory_existence_of_inverses}
Let $M$ be a closed manifold satisfying $\oplus_{i>0} H^{2i}(M,\QQ)=0$ (resp. $\oplus_{i>0} H^{4i}(M,\QQ)=0$). Let $S$ be a set of complex (resp. real) vector bundles over $M$ satisfying the following properties:
\begin{enumerate}
\item $S$ contains all trivial bundles,
\item the sum of two bundles in $S$ belongs to $S$, and
\end{enumerate}
Then every bundle in $S$ has an inverse in $S$.
\end{lemma}

\begin{proof}
Using the Atiyah-Hirzebruch spectral sequence for K-theory it follows that the condition $\oplus_{i>0} H^{2i}(M,\QQ)=0$ is equivalent to the reduced K-theory ring of $M$ being finite (see \cite[Section~2.4]{AH61}). The latter is equivalent to the fact that the set of stable classes of complex bundles over $M$ is finite. In particular, for any complex bundle $E$ over $M$ there exist integers $n,k$ such that $nE\oplus k$ is a trivial bundle, where $nE$ denotes the Whitney sum $E\oplus\dots\oplus E$ of $n$ copies of $E$. Thus $F\defeq (n-1)E\oplus k$ is an inverse for $E$. The assumptions on $S$ imply that $F\in S$, as desired. The same discussion applies in the case of real bundles under the condition $\oplus_{i>0} H^{4i}(M,\QQ)=0$ (see the proof of \cite[Corollary~2.11]{AGZ} for more information).
\end{proof}

\begin{lemma}\label{LEM:H2_existence_of_inverses}
Let $M\approx G\bq H$ be a closed biquotient with $G$ and $M$ simply connected and satisfying 
$$\oplus_{i>0} H^{2i}(M,\QQ)=H^{2}(M,\QQ)=\QQ.$$ 
Then any complex biquotient bundle has an inverse among biquotient bundles.
\end{lemma}

\begin{proof}
Let $E$ be a complex vector bundle over $M$. Recall that the Chern character induces an isomorphism $\K(M)\otimes\QQ\to H^{\text{even}}(M,\QQ)$ \cite[Section~2.4]{AH61}. By the cohomological assumption on $M$, the only Chern class that survives is $c_1(E)$. 

On the other hand, if $M=G\bq H$ then, as $M$ is simply connected, $\pi_2(M)\cong H_2(M)$ is non-trivial, so $H$, up to cover, must split off a circle factor \cite[Proposition 3.3]{De17}. From Proposition \ref{prop:allbundles} below, all complex line bundles over $G\bq H$ are biquotient bundles, and there is precisely one such bundle for any element in $H^{2}(M,\ZZ)$. Let $L$ the line bundle with $c_1(L)=-c_1(E)$. Then $c_1(E\oplus L)= 0$ and hence the Chern character maps the K-theory class of $E\oplus L$ to the trivial element in $H^{\text{even}}(M,\QQ)$. Hence, the stable class of $E\oplus L$ corresponds to a torsion element in $\K(M)$, and consequently there exist integers $n,k$ for which $n(E\oplus L)\oplus k= nE\oplus nL \oplus k$ is a trivial bundle. Thus $(n-1)E\oplus nL\oplus k$ is an inverse for $E$ and also a biquotient bundle, as desired.
\end{proof}


\begin{proof}[Proof of Theorem~\ref{thm:biqinv_sufficient_conditions}]
The first two statements follow directly from Lemma~\ref{LEM:finite_Ktheory_existence_of_inverses} applied to the set $S$ of real (resp. complex) biquotient bundles together with the fact that $S$ is clearly closed under taking Whitney sums. The third statement is precisely Lemma~\ref{LEM:H2_existence_of_inverses}.
\end{proof}

\begin{proof}[Proof of Corollary~\ref{COR:dimensions_2_3_5}]
By the classification in \cite{To02,De14,Pa04}, such a biquotient is diffeomorphic to one of the following: $S^2,S^3,S^5,S^2\times S^3, SU(3)/SO(3)$ or the non-trivial $S^3$-bundle over $S^2$, denoted by $S^3\tilde\times S^2$.  Their topology is well known. The real case in the corollary follows from Lemma~\ref{LEM:finite_Ktheory_existence_of_inverses}. The complex case follows from Lemma~\ref{LEM:finite_Ktheory_existence_of_inverses} in the cases of $S^3, SU(3)/SO(3)$ and from Lemma~\ref{LEM:H2_existence_of_inverses} in the cases of $S^2,S^2\times S^3,S^3\tilde\times S^2$.
\end{proof}

We conclude this section with a proof of Proposition~\ref{PROP}.

\begin{proof}[Proof of Proposition~\ref{PROP}] 
Recall that the Grothendieck group of a set of vector bundles $S$ is defined as $\K(S)\defeq\{A-B : A,B\in S\}$, where $A_1-B_1=A_2-B_2$ if and only if there exists $m$ for which the bundles $A_1\oplus B_2\oplus m$ and $A_2\oplus B_1\oplus m$ are isomorphic. We will denote the image of an element $A-B\in\K (S)$ under $\imath\colon \K(S)\to \K(M)$ simply by $A-B$. 

Suppose first that $\imath$ is surjective. Then any $E\in \Vect(M)$ can be written as $E=A-B$ in $K(M)$ for certain $A,B\in S$. By taking an inverse $B'\in S$ of $B$ (i.e.~$B\oplus B'=n$) and adding the trivial element $0=B' - B'$ we get 
$$
E=E+B'-B'=A-B+B' - B'= (A+B')-(B+B')=A+B'-n.
$$
It follows that $E\oplus n\oplus m=A\oplus B'\oplus m$ for some $m$. The bundle in the right-hand-side belongs to $S$ by assumption and hence the claim follows for $k\defeq n+m$.

Suppose now that for every $E\in\Vect(M)$ there is $k$ for which $E\oplus k\in S$. Take any element $E_1 - E_2\in \K(M)$. By assumption, there is $k_i$ for which $F_i\defeq E_i\oplus k_i\in S$, for $i=1,2$. Clearly $E_1 - E_2=F_1 -F_2\in K(S)$ and hence the claim follows.
\end{proof}

\section{Biquotients bundles with no inverses}\label{sec:noinv}

We now move towards proving Theorem \ref{thm:biqinv}.  Given a manifold $M$, we say $M$ has Property $(\ast)$ if $$ \text{for any }x_1,..., x_n \in H^2(M),\quad \sum  x_i^2 = 0 \text{ if and only if } x_1 = x_2 = .... =x_n=0,$$
where we use the notation $H^k(M)$ as an abbreviation for $H^k(M,\mathbb{Z})$. Our interest in Property $(\ast)$ stems from the following proposition.

\begin{prop}\label{prop:noreverse} 
Suppose $M = G\bq T^k$ is a closed simply connected manifold satisfying property $(\ast)$.  Then in both the complex and real cases, the only biquotient bundles over $M$ which have a biquotient inverse are the trivial vector bundles.
\end{prop}

\begin{proof}Recall from e.g. \cite{BtD85}, that all irreducible complex representations of $T^k$ are $1$ dimensional over $\mathbb{C}$ and all real representations of $T^k$ are the realification of a complex representation, possibly summed with a trivial rank $1$ real vector bundle.  It follows that  all complex biquotient bundles over $M$ split as a sum of line bundles and that every real biquotient bundle over $M$ is the realification of a complex bundle over $M$, possibly summed with a trivial rank $1$ real vector bundle.

We shall show that if a sum of complex line bundles over $M$ is trivial then each of the line bundles in the sum is trivial. This implies that if the sum $E\oplus F$ of two biquotient bundles is trivial then so are $E$ and $F$, as claimed.

Suppose $L_1,...,L_n$ are line bundles over $M$ and set $L = \bigoplus L_i$.  Computing the first and second Chern classes via the Whitney sum formula, we find $c_2(L) = \sum_{1\leq i<j\leq n} c_1(L_i)c_1(L_j)$ and $c_1(L)=\sum_{i=1}^n c_1(L_i)$. Thus, $$c_1(L)^2 - 2c_2(L) = \sum_{i=1}^n c_1(L_i)^2.$$
If $L$ is trivial, the left hand side vanishes, and so by Property $(\ast)$, we conclude $c_1(L_i) = 0$ for all $i$.  Since line bundles are classified by their first Chern class, each $L_i$ must then be trivial.

Lastly, observe that if $rL$ denotes the realification of $L$, then its first Pontryagin class equals $p_1(rL) = c_1(L)^2 - 2c_2(L) = \sum c_1(L_i)^2$, so $p_1(rL)$ is non-trivial if at least one $L_i$ is non-trivial.
\end{proof}

Of course, in order to use Proposition \ref{prop:noreverse}, we must construct biquotients with Property $(\ast)$ and verify that they admit non-trivial biquotient vector bundles.  The latter statement is a consequence of the following proposition, Proposition \ref{prop:allbundles}.  This proposition is likely not new, but we could not locate it in the literature (see \cite[p.~227]{WZ90} or \cite[pp.~473-474]{KS91} for particular cases of it).

Given a subgroup $S^1\subseteq T^k$, we call a subgroup $T^{k-1}\subseteq T^k$ \textit{complementary} to $S^1$ if $S^1\cdot T^{k-1} = T^k$ and $S^1$ and $T^{k-1}$ intersect only at the identity.  If $S^1$ and $T^{k-1}$ are complementary, then every element in $T^k$ has a unique expression of the form $uw$ with $u\in S^1$ and $w\in T^{k-1}$.

\begin{prop}\label{prop:allbundles}
Suppose $M$ is a $2$-connected manifold and a torus $T^k$ acts freely with quotient $N = M/T^k$ and suppose $L$ is any complex line bundle over $N$.  Then there is a subgroup $S^1\subseteq T^k$ with complementary $T^{k-1}$ for which $L$ is isomorphic to $M/T^{k-1}\times_{S^1}\mathbb{C}$ where $S^1$ acts on $\mathbb{C}$ via $v\ast z = v^d z$ for some $ d\in \mathbb{Z}$. Equivalently, $L$ is isomorphic to $M\times_{T^k}\mathbb{C}$ where $T^k=S^1\cdot T^{k-1}$ acts on $\mathbb{C}$ via $(v(w_1,\dots,w_{k-1}))\ast z = v^d z$.
\end{prop}

Before proving this proposition, we note the following obvious corollary.

\begin{corollary}   
If $G$ is a simply connected compact Lie group and $T^k$, $k\geq 1$, acts via a biquotient action, then any line bundle $L$ over $G\bq T^k$ is a biquotient bundle. In particular, since line bundles are classified by $H^2(G\bq T^k)\cong \mathbb{Z}^k$, there are infinitely many pairwise non-isomorphic non-trivial biquotient bundles.
\end{corollary}


\begin{proof}[Proof of Proposition \ref{prop:allbundles}]
Because complex line bundles are classified by their first Chern class, it is sufficient to show every element of $H^2(N)$ can be realized as the first Chern class of a bundle of the form $M/T^{k-1}\times_{S^1} \mathbb{C}$ for some complementary $S^1, T^{k-1}$ in $T^k$. Further, every line bundle is naturally associated to a principal $S^1$-bundle, and under this association, the first Chern class of the line bundle agrees with the Euler class of the principal bundle.  Thus, it is enough to show that every element $x\in H^2(N)$ can be realized as the Euler class of a principal circle bundle of the form $M/T^{k-1}\times_{S^1} S^1$ where $S^1$ acts on itself by rotations at some speed.

The hypothesis on $M$ together with the Serre spectral sequence for the principal $T^k$-bundle $T^k\rightarrow M\rightarrow N$ imply that there is a natural isomorphism $H^1(T^k)\cong H^2(N)$.  We write  $v_i$ for the coordinates on $T^k$ and $H^\ast(T^k)\cong \Lambda_\mathbb{Z}(z_1,...,z_k)$ with each $z_i$ corresponding to the $i$-th factor of $T^k$. Then $H^2(N) = \mathbb{Z}^k$ is generated by $dz_1,...,dz_k$, where $d$ denotes the differential for the spectral sequence associated to the bundle.


Now, let $x = \sum b_i (dz_i)$ be any element of $H^2(N)$.  We may obviously assume $x\neq 0$.  Moreover, we may assume $\gcd(b_1,...,b_k) = 1$, for if we have a bundle of the form $M/T^{k-1}\times_{S^1} S^1$ with Euler class $x/\gcd(b_1,...,b_k)$, we may obtain $x$ as an Euler class by having the $S^1$ act on itself by $\gcd(b_1,...,b_k)$-fold rotation.

We now define $T_x^{k-1}$ and $S^1_x$.  Because $\gcd(b_1,...,b_k) = 1$, there is a matrix $B\in Gl(k,\mathbb{Z})$ whose first row is $(b_1,...,b_k)$, see e.g. \cite{Sh60}. Then $B$ defines an isomorphism $B:T^k\rightarrow T^k$ given by $$B(v_1,..., v_k) = (v_1^{B_{11}} v_2^{B_{12}} ... v_k^{B_{1k}}, ..., v_1^{B_{k1}} v_2^{B_{k2}}... v_k^{B_{kk}}).$$  We set
\begin{align*}
S_x^1 &= \{ B^{-1}(v_1,1,...,1) : (v_1,1,...,1)\in T^k\}\\
T_x^{k-1} &= \{B^{-1}(1,v_2,...,v_k) : (1,v_2,...,v_k)\in T^k\}
\end{align*}
Clearly $S_x^1$ and $T_x^{k-1}$ are complementary.


We now compute the Euler class of $M_x =M/T^{k-1}_x \times_{S^1_x} S^1$ where $S^1_x$ acts on $S^1$ as $v\ast w = v w$.  Define a map $\phi:T^k\rightarrow S^1$ by $\phi(v_1,..., v_k) =\pi_1(B(v_1,...,v_k)) = v_1^{b_1}...v_k^{b_k}$, where $\pi_1:T^{k}\rightarrow S^1$ is projection onto the first coordinate, $\pi_1(v_1,...,v_k) = v_1$.   Also, define $\psi:M\rightarrow M_x$ by $\psi(m) = [ mT^{k-1}_x,1]$ and consider the diagram of fiber bundles
$$\begin{diagram} T^k & \rTo^{\phi} & S^1\\ \dTo^{i_{T^k}}  & & \dTo^{i_{S^1}}\\ M & \rTo{\psi} & M_x\\ \dTo & & \dTo \\ N & \rTo{Id_N} & N\end{diagram}$$
where the bottom vertical maps are the obvious projections and $i_{T^k}, i_{S^1}$ are the inclusion of a fiber.  Fixing $m_0\in M$ and considering the corresponding point $\psi(m_0)\in M_x$, we have $i_{T^k}(v_1,..,v_k) = (v_1,...,v_k)m_0$ and $i_{S^1}(v) = [m_0 T^{k-1}_x,v]$.

We claim the diagram commutes.  Indeed, the bottom square obviously commutes.  For the top square, we 
then observe the top square commutes if and only if for any $(v_1,...,v_k)\in T^k$,  $$[(v_1,...,v_k)m_0T_x^{k-1},1] =  [m_0T^{k-1}_x, v_1^{b_1}...v_k^{b_k}].$$  
Writing $(v_1,...,v_k) = B^{-1}(w_1,1,...,1) B^{-1}(1,w_2,...,w_k)$ for some $(w_1,w_2,...,w_k)\in T^k$, we find $w_1 = \pi_1(B(v_1,...,v_k)) = v_1^{b_1} ... v_k^{b_k}$, so 
\begin{align*} [(v_1,...,v_k)m_0 T_x^{k-1},1] &= [B^{-1}(w_1,1,...,1) m_0 T_x^{k-1},1]\\ &= [ m_0T_x^{k-1}, w_1]\\ &= [m_0 T_x^{k-1}, v_1^{b_1}...v_k^{b_k}].
\end{align*}
Thus, there is a map of spectral sequences for the two bundles.  Writing $H^\ast(S^1) = \Lambda_\mathbb{Z}(z)$, and using $d$ and $d_x$ for the differential for the $M$ and $M_x$ bundles respectively, we now see that the Euler class $e(M_x)$ of $M_x$ satisfies $e(M_x) = d_x(z) = d\phi^\ast(z)$.  But, $\phi(v_1,...,v_k) = v_1^{b_1}...v_k^{b_k}$, so obviously $d\phi^\ast(z) = d\sum b_i z_i = x$.
\end{proof}


We now find examples of biquotients of the form $G\bq T^k$ with Property $(\ast)$.

\begin{prop}\label{PROP:CP2CP2_and_Eschenburg_satisfy_ast}
The space $\mathbb{C}P^2\sharp \mathbb{C}P^2$ (which is a biquotient of the form $(S^3)^2\bq T^2$ \cite[p.~404]{To02}) as well as the inhomogeneous biquotient $SU(3)\bq T^2$ both have Property $(\ast)$.
\end{prop} 

\begin{proof}For $\mathbb{C}P^2\sharp \mathbb{C}P^2$, recall that $H^\ast(\mathbb{C}P^2\sharp \mathbb{C}P^2) \cong \mathbb{Z}[u,v]/\langle u^2 - v^2, uv\rangle$ with both $|u|=|v| = 2$.  Then given elements $x_i = a_i u + b_i v$, one easily computes that $\sum x_i^2 = \sum (a_i^2 + b_i^2)u^2$, so obviously vanishes if and only if all $x_i = 0$.

For $SU(3)\bq T^2$, recall from \cite{Es92} that $H^\ast(SU(3)\bq T^2)\cong \mathbb{Z}[u,v]/\langle v^2 - u^2 - uv, u^3\rangle$ with $|u|=|v| = 2$. Given elements $x_i = a_i u + b_i v$, we compute that $\sum x_i^2 = \sum_i a_i^2 u^2 + b_i^2 v^2  + 2a_ib_i u v = \sum_i (a_i^2 + b_i^2) u^2 + (2a_i b_i + b_i^2)uv$.  From the $u^2$ component, it follows that $\sum x_i^2 = 0$ if and only if $a_i = b_i = 0$ for all $i$, i.e., if and only if $x_i = 0$ for every $i$.
\end{proof}

We next construct in each even dimension $n=2k\geq 4$ biquotients of the form $(S^3)^k\bq T^k$ having Property $(\ast)$.

Suppose $k\geq 2$.  We call a $k\times k$ integer matrix $A$ \textit{admissible} if 
$$A_{ij} = \begin{cases} 1  & i=j \text{ or } (i,j) = (2,1)\\ 2 & (i,j) = (1,2) \\ 0 & i < j \text{ with } (i,j)\neq (1,2)\end{cases}.$$   
For example, a typical $5\times 5$ example has the form
$$A= \begin{bmatrix} 1 &2 & 0 & 0 & 0\\ 1 & 1 & 0 & 0 & 0\\ A_{3,1} & A_{3,2} & 1  & 0 &0\\  A_{4,1} & A_{4,2} & A_{4,3} & 1 & 0\\ A_{5,1} & A_{5,2} & A_{5,3} & A_{5,4} & 1\end{bmatrix}.$$
Given an admissible $A$, we define a biquotient action of $T^k$ on $(S^3)^k$ as follows.  Viewing $S^3\subseteq \mathbb{C}^2$, on the $i$-th coordinate $(a_i,b_i)$ of $(S^3)^k$ we have $$(w_1,...,w_k)\ast(a_i,b_i) = (w_i a_i, w_1^{A_{i1}} \cdot ...\cdot w_k^{A_{ik}} b_i).$$

\begin{prop}\label{prop:free} For any admissible $A$, the above action is free.
\end{prop}

\begin{proof}We prove this by induction on $k$.  The case $k=2$ is \cite[Proposition 3.2]{De14}.

Suppose $A$ is any $k\times k$ admissible matrix with $k\geq 3$.  Because of the last column of zeros, any element $(1,...,1,w_k)\in T^k$ acts only on the last $S^3$ factor of $(S^3)^k$, and on this $S^3$ factor, it acts as the Hopf map with quotient $S^2$.  Writing $T^{k-1}$ for the set of points of the form $(w_1,...,w_{k-1},1)$, we may therefore view $(S^3)^3\bq T^k$ as $(S^3)^{k-1}\times_{T^{k-1}} S^2$.  The $T^{k-1}$ action on $(S^3)^{k-1}$ is given by the $(k-1)\times (k-1)$ matrix $A'$ formed by deleting the last row and last column of $A$.  In particular, $A'$ is admissible and so the $T^{k-1}$ action on $(S^3)^{k-1}$ is free, and thus, the $T^k$ action on $(S^3)^k$ is free.
\end{proof}

Since the action corresponding to an admissible matrix $A$ is free, the quotient space is a manifold which we denote by $R(A)$.  Then $R(A)$ is simply connected and has dimension $2k$.  Further, from the proof of Proposition \ref{prop:free}, we see that if $k\geq 3$, $R(A)$ naturally has the structure of an $S^2$-bundle over $R(A')$, where $A'$ is obtained from $A$ be deleting the last row and last column.  When $k = 2$, the conditions force $A = \begin{bmatrix} 1 & 2\\ 1 & 1\end{bmatrix}$.  In this case, $R(A)\cong \mathbb{C}P^2\sharp \mathbb{C}P^2$ \cite[Proposition 3.3]{De14}.

\begin{prop}\label{PROP:RA_satisifes_ast}
Suppose $A_{i1}\neq 0$ for all $1\leq i\leq k$.  Then $R(A)$ satisfies $(\ast)$.
\end{prop}

\begin{proof}Following \cite[Section 2.4 and Proposition 4.26]{De17}, it is easy to see that $H^\ast(R(A)) \cong \mathbb{Z}[u_1,...,u_k]/I$ where all the $u_i$ have degree $2$ and $I$ is the ideal $I =\langle \sum_{j=1}^k A_{ij} u_iu_j \rangle $ with $1\leq i\leq k$. Hence modulo $I$, we have $u_1^2 = -2u_1 u_2$ and $u_i^2 = -\sum_{j\neq i} A_{ij} u_i u_j$ for $i\geq 2$.  In particular $H^4(R(A))$ has as a basis all $u_i u_j$ with $1\leq i < j \leq k$.

Now, suppose $x_j = \sum_{i=1}^k \alpha_{ij} u_i$ and that $\sum_j x_j^2 = 0$.  A simple computation shows that the $u_1u_2$ component of $x_j^2$ is $-2\alpha_{1j}^2 - \alpha_{2j}^2 + 2\alpha_{1j} \alpha_{2j} = -\alpha_{1j}^2 - (\alpha_{1j} - \alpha_{2j})^2$. Thus, the $u_1 u_2$ component of the equation $\sum_j  x_j^2 = 0$ is $$\sum_j -\alpha_{1j}^2 - (\alpha_{1j}-\alpha_{2j})^2 = 0,$$ so clearly $\alpha_{1j} = \alpha_{2j} = 0$ for every $j$.

Fix any $i\geq 3$.  Since $\alpha_{1j} = 0$ for every $j$, it now follows that the $u_1 u_i$ component of $(x_j)^2$ is $-A_{i1} \alpha_{ij}^2$, so the $u_1u_i$ component of the equation $\sum_j  x_j^2 =0$ is $-A_{i1}\sum_j\alpha_{ij}^2 = 0$.  As $A_{i1}\neq 0$ by hypothesis, we find that $\alpha_{ij} = 0$ for all $j$. Hence, $\alpha_{ij} = 0$ for all $i\geq 3$ and all $j$. That is, all $x_j$ are zero.
\end{proof}

\begin{remark}\label{REM:cohomology_of_CP2CP2}
Note that the cohomology ring $\mathbb{Z}[u_1,u_2]/\langle u_1^2 +2u_1 u_2,u_2^2+ u_1 u_2 \rangle$ of $R(A)$ with $k=2$ is isomorphic to the usual cohomology ring $\mathbb{Z}[u,v]/\langle u^2 - v^2, uv\rangle$ of $\mathbb{C}P^2\sharp \mathbb{C}P^2$ from Proposition~\ref{PROP:CP2CP2_and_Eschenburg_satisfy_ast}. An isomorphism can be given by identifying $u = u_1 + u_2$ and $v = u_2$.
\end{remark}


\begin{prop}\label{PROP:infiniteness_RA}  
For each fixed $k\geq 3$, there are infinitely many homotopy types among $R(A)$ having Property $(\ast)$.
\end{prop}

\begin{proof} For each odd prime $p\geq 3$, we let $R(p)$ denote $R(A)$ where $A$ has first column $(1,1,p,...,p)^t$, second column $(2,1,0,...,0)^t$, and for all $3\leq i\leq k$, with the $i$-th column of the form $(0,...,0,1,0...,0)^t$ with a $1$ in the $i$-th position.  For example, when $k = 5$, $$A= \begin{bmatrix} 1 &2 & 0 & 0 & 0\\ 1 & 1 & 0 & 0 & 0\\ p & 0 & 1  & 0 &0\\  p & 0 & 0 & 1 & 0\\ p & 0 & 0 & 0 & 1\end{bmatrix}.$$  By Proposition~\ref{PROP:RA_satisifes_ast}, all the $R(p)$ satisfy Property $(\ast)$. We claim that $H^\ast(R(p))$ is not isomorphic to $H^\ast(R(p'))$ for distinct odd primes $p,p'$.

The cohomology ring for $R(p)$ is then given via $H^*(R(p)) \cong \ZZ [ u_1, ... , u_k]/  I$ where all $|u_i| = 2$ and $I$ is the ideal $\langle u_1^2 + 2u_1 u_2,  u_2^2 + u_1 u_2, u_i^2 + p u_1 u_i\rangle$ for $3\leq i\leq k$. In particular, we observe that $u_2^2 = (u_1 + u_2)^2$ and $u_2 (u_1 + u_2)  =0$ for any $k\geq 2$ and any $p$. We denote by $U_{12}$ the subgroup of $H^2(R(p))$ generated by $u_1$ and $u_2$.


We claim that if $x,y\in H^2(R(p))$ are primitive elements and satisfy $x^2 = y^2$ and $xy=0$, then $x,y\in U_{12}$. To see this, write $x =\sum \alpha_i u_i$ and $y = \sum \beta_i u_i$ with $\gcd(\alpha_1,...,\alpha_k) = \gcd(\beta_1,...,\beta_k) = 1$.  The $u_iu_j$ components of the equations $x^2 =y^2$ and $xy = 0$ are

 \begin{tabular}{c|c|c}
Component & $x^2 = y^2$ & $xy=0$\\

\hline

$u_1 u_i$, $i\geq 3$ & $-p\alpha_i^2 + 2\alpha_i \alpha_1 = -p\beta_i^2 + 2\beta_i \beta_1$ & $\alpha_1\beta_i + \alpha_i\beta_1 - p\alpha_i\beta_i = 0$\\

$u_iu_j$, $i,j\geq 2$, $i\neq j$ & $2\alpha_i \alpha_j = 2\beta_i\beta_j$ & $\alpha_i \beta_j + \alpha_j\beta_i=0$.\\
\end{tabular}

Assume for a contradiction that, without loss of generality, $\alpha_3\neq 0$.  By replacing $x$ with $-x$, we may assume $\alpha_3 > 0$.   If $\beta_3 = 0$, then components of the equation $xy = 0$ easily yield $\beta_j = 0$ for all $j$.  This contradicts primitivity of $y$.  So, $\beta_3\neq 0$ and we may likewise assume $\beta_3 > 0$.

We will now show that $\alpha_3=\beta_3$.  Indeed, if $\alpha_3\neq \beta_3$, then up to swapping $\alpha_3$ and $\beta_3$, there is a prime $q$ and an integer $k \geq 1$ for which $q^k$ divides $\alpha_3$ but not $\beta_3$.  Reducing the equation $\alpha_3\beta_j + \alpha_j\beta_3=0$ mod $q^k$, we see that $q|\alpha_j $ for all $j\geq 2$.  Likewise, from $\alpha_1\beta_3 + \alpha_3\beta_1 -p\alpha_3\beta_3 = 0$, we deduce that $q|\alpha_1$ as well.   This contradicts primitivity of $x$, establishing that $\alpha_3=\beta_3$.

Returning to the equations $\alpha_i \alpha_3 = \beta_i\beta_3$, they now obviously imply that $\alpha_i = \beta_i$ for all $i\geq 2$.  Further, from the equation $-p\alpha_3^2 + 2\alpha_3\alpha_1 = -p\beta_3^2 + 2\beta_3\beta_1$, we now deduce that $\alpha_1 = \beta_1$.

Thus, if $\alpha_3\neq 0$, we must have $x = y$.  But then the equation $0 = xy = x^2$ contradicts the fact that $H^\ast(R(p))$ has Property $(\ast)$ by Proposition \ref{PROP:RA_satisifes_ast}.  Thus, the claim is established.

Now, suppose there is an isomorphism $\phi:H^\ast(R(p))\rightarrow H^\ast(R(p'))\cong \mathbb{Z}[u_1',...,u_k']/I'$ and let $U'_{12}$ be the subgroup of $H^2(R(p'))$ generated by $u_1'$ and $u_2'$.  Then the claim implies that $\phi(U_{12}) = U'_{12}$.   Thus $\phi(u_1) = \alpha_1 u_1' + \alpha_2 u_2'$ and $\phi(u_3) = \sum \gamma_i u_i'$ for some integers $\alpha_i,\gamma_i$.  Because $u_3^2 = -p u_1 u_3$, we must have $\phi(u_3)^2 = -p\phi(u_1)\phi(u_3)$.

For $i,j\geq 3$ with $i\neq j$, the coefficient of $u'_i u'_j$ in $\phi(u_3)^2=(\sum \gamma_i u_i')^2$ is $2\gamma_i\gamma_j$.  For $\phi(u_3)^2=-p\phi(u_1)\phi(u_3)$, the same coefficient is obviously $0$.  It now follows that $\gamma_i\neq 0$ for at most one $i\geq 3$.  On the other hand, since $u_3\notin U_{12}$, $\phi(u_3)\notin U'_{12}$, so some $\gamma_i \neq 0$ for at least one $i\geq 3$.  Thus, there is some unique $i$ for which $\phi(u_3) = \gamma_1 u_1' + \gamma_2 u_2' + \gamma_i u_i'$.

Since $\gamma_i\neq 0$, it follows that the equation $\phi(u_3)^2 = -p\phi(u_1)\phi(u_3)$ is equivalent to the following three equations, corresponding to the components $u_1'u_2',u_1'u_3',u_2'u_3'$, respectively.
\begin{align*} -\gamma_1^2 - (\gamma_1-\gamma_2)^2 &= -p(\alpha_1\gamma_2 + \alpha_2\gamma_1 -2\alpha_1\gamma_1 - \alpha_2\gamma_2)\\ 
-p' \gamma_i + 2\gamma_1 &= -p\alpha_1\\
2\gamma_2 &= -p\alpha_2.\end{align*}
Reducing all three equations mod $p$, one easily sees that $p|\gcd(\gamma_1, \gamma_2, \gamma_i)$, so $\phi(u_3)$ is not primitive.  However, $\phi$ is a isomorphism and $u_3$ is primitive, so we have a contradiction.
\end{proof}

We are now ready to prove Theorems \ref{thm:biqinv} and \ref{thm:improve}.  Then $n=4$ case of Theorem \ref{thm:biqinv} is provided by $(S^3)^2\bq T^2\cong \mathbb{C}P^2\sharp \mathbb{C}P^2$, see Propositions \ref{PROP:CP2CP2_and_Eschenburg_satisfy_ast} and \ref{prop:noreverse}.  The case $n\geq 6$ of Theorem \ref{thm:biqinv} is a special case of Theorem \ref{thm:improve}, so we just focus on the latter theorem.

\begin{proof}[Proof of Theorem~\ref{thm:improve}]

The existence of infinitely many homotopy types in each even dimension $n\geq 6$ is provided by the biquotients $R(p)$ with $p\geq 3$ prime thanks to Propositions \ref{prop:noreverse}, \ref{PROP:RA_satisifes_ast} and \ref{PROP:infiniteness_RA}.

For odd $n = 2k+1 \geq 9$, the examples are likewise provided by $R(p)\times S^3 \cong (S^3)^{k}\bq T^{k-1}$.  By the K\"unneth formula, $R(p)\times S^3$ has Property $(\ast)$ since $R(p)$ does.  Further, the proof that $H^\ast(R(p))\not\cong H^\ast(R(p'))$ for $p\neq p'$ only used the cup product from $H^2(R(p))$ to $H^4(R(p))$, so $R(p)\times S^3$ is homotopy equivalent to $R(p')\times S^3$ if and only if $p = p'$.

Lastly, in dimension $7$, we use biquotients of the form $Q(s,t):=(S^3\times (S^3)^2)\bq T^2$ where $T^2$ acts on $(S^3)^2$ with quotient $\mathbb{C}P^2\sharp \mathbb{C}P^2$ and $T^2$ acts on $S^3\subseteq \mathbb{C}^2$ via $(z,w)\ast (a,b) = (z^s w^t a, b)$. As shown in \cite[Proposition 4.35]{De17}, $H^\ast(Q(s,t)) \cong H^\ast(\mathbb{C}P^2\sharp \mathbb{C}P^2 \times S^3)$, independent of $s$ and $t$.  In particular, these have Property $(\ast)$ and hence Proposition~\ref{prop:noreverse} applies.  Moreover, $p_1(Q(s,t)) = \pm(6-s^2 - (s-t)^2)u$, where $u \in H^4(Q(s,t))\cong \mathbb{Z}$ is a generator.  Since rational Pontryagin classes are homeomorphism invariants \cite{No66} we have infinitely many homeomorphism types among these manifolds.
\end{proof}



\begin{remark} We point out that the formula for $p_1$ in \cite[Proposition 4.35]{De17} contains a misprint.  In the notation of that proposition, the correct formula is $p_1((S^3)^3\bq T^2)= \sum_{i=1}^3 (k_i-m_i)^2 t^2 + (l_i-n_i)^2u^2$.  In addition, the ideal $J$ is missing a generator $x^2$.
\end{remark}

\end{document}